\documentclass[11pt]{article}
\usepackage{amsthm, amsmath, amssymb, amsfonts, url, booktabs, tikz, setspace, fancyhdr, bm}
\usepackage{hyperref}
\usepackage{amsthm}
\usepackage{geometry}
\usepackage{hyperref, enumerate}
\usepackage[shortlabels]{enumitem}
\usepackage[babel]{microtype}
\usepackage[english]{babel}
\usepackage[capitalise]{cleveref}
\usepackage{comment}
\usepackage{bbm}
\usepackage{csquotes}
\usepackage{mathabx}
\usepackage{tikz}
\usepackage{graphicx}
\usepackage{float}
\usepackage{xcolor}

\counterwithin{figure}{section}

\newtheorem{theorem}{Theorem}[section]

\newtheorem{lemma}[theorem]{Lemma}

\theoremstyle{definition}

\newtheorem*{defn-non}{Definition}

\newlist{Case}{enumerate}{2}
\setlist[Case, 1]{%
    label           =   {\bfseries Case \arabic*.},
    labelindent=1em ,labelwidth=1.3cm, labelsep*=1em, leftmargin =!
}
\setlist[Case, 2]{%
    label           =   {\bfseries Subcase \arabic{Casei}.\arabic*.},
    labelindent=-1em ,labelwidth=1.3cm, labelsep*=1em, leftmargin =!
}

\newcommand{\ceil}[1]{\lceil #1\rceil}

\usepackage{todonotes}

\title{All rectangles exhibit canonical Ramsey property}

\author{
Gennian Ge\thanks{School of Mathematical Sciences, Capital Normal University, Beijing 100048, China. Email: gnge@zju.edu.cn. Gennian Ge is supported by the National Key Research and Development Program of China under Grant 2020YFA0712100, the National Natural Science Foundation of China under Grant 12231014, and Beijing Scholars Program.}
\and 
Yang Shu\thanks{School of Mathematical Sciences, University of Science and Technology of China, Hefei, Anhui 230026, China. Email: shuyangyyyy@mail.ustc.edu.cn}
\and
Zixiang Xu\thanks{Extremal Combinatorics and Probability Group (ECOPRO), Institute for Basic Science (IBS), Daejeon 34126, South Korea. Email: zixiangxu@ibs.re.kr. Supported by IBS-R029-C4.}
}

\date{}

\begin{document}
\maketitle

\begin{abstract}

In a seminal work, Cheng and Xu proved that for any positive integer \(r\), there exists an integer \(n_0\), independent of \(r\), such that every \(r\)-coloring of the \(n\)-dimensional Euclidean space \(\mathbb{E}^n\) with \(n \ge n_0\) contains either a monochromatic or a rainbow congruent copy of a square. This phenomenon of dimension-independence was later formalized as the canonical Ramsey property by Gehe\'{e}r, Sagdeev, and T\'{o}th, who extended the result to all hypercubes, and to rectangles whose side lengths \(a\), \(b\) satisfy \((\frac{a}{b})^2\) is rational. They further posed the natural problem of whether every rectangle admits the canonical Ramsey property, regardless of the aspect ratio.

In this paper, we show that all rectangles exhibit the canonical Ramsey property, thereby completely resolving this open problem of Gehe\'{e}r, Sagdeev, and T\'{o}th. Our proof introduces a new structural reduction that identifies product configurations with bounded color complexity, enabling the application of simplex Ramsey theorems and product Ramsey amplification to control arbitrary aspect ratios.
\end{abstract}

\section{Introduction}

Euclidean Ramsey theory, introduced by Erd\H{o}s, Graham, Montgomery, Rothschild, Spencer, and Straus~\cite{1973JCTA} in 1975, investigates the inevitable appearance of geometric patterns in Euclidean space under arbitrary colorings. Originating from classical Ramsey theory and graph theory, it explores the interplay between combinatorial and geometric structures in continuous domains. Over the past fifty years, numerous profound results have been established, exemplified by notable works~\cite{2019DCGCONLON,conlon2022more,2024Yip,2026JCTA,1990JAMS,2004FranklRodl,2025EUJC,1980GrahamJCTA,2009Combinatorica,SHADER1976385} and the references therein. We also refer the readers to the excellent textbook~\cite{1997HANDBOOK} for a comprehensive overview.

In parallel, Gallai~\cite{1967Gallai} initiated the study of the Gallai-Ramsey problem, which focuses on the existence of monochromatic or rainbow substructures in edge-colored complete graphs and more general colored systems. The Gallai-Ramsey problem has since evolved into a rich and active area of research, with significant progress on extremal bounds and structural characterizations~\cite{2020SIDMABalogh,2010JGTGyarfas,2020JGTLiu}.

Recently, Mao, Ozeki, and Wang~\cite{2022arxivEGR} introduced the \emph{Euclidean Gallai-Ramsey problem}, which combines ideas from Euclidean Ramsey theory and Gallai-Ramsey theory. Given two finite geometric configurations \(K_1\) and \(K_2\), we write
\[
\mathbb{E}^n \overset{r}{\rightarrow} (K_1; K_2)_{\mathrm{GR}}
\]
to denote the following statement: for every coloring \(\chi: \mathbb{E}^n \to [r]\), there exists either a monochromatic copy of \(K_1\) or a rainbow copy of \(K_2\). 

Cheng and Xu~\cite{2025DCGChengXu} studied this problem for various fundamental configurations, with particular focus on triangles, squares, and lines. They discovered that many configurations exhibit a remarkable \emph{dimension-independence phenomenon}: the existence of monochromatic or rainbow copies is guaranteed once the dimension exceeds a threshold that depends only on the configuration, not on the number of colors. For example, they proved that for any right triangle, acute triangle, or square \(K\), there exists an integer \(n_0\), independent of \(r\), such that \(\mathbb{E}^n \overset{r}{\rightarrow} (K; K)_{\mathrm{GR}}\) holds for all \(n \ge n_0\).

To formalize this dimension-independence phenomenon, Gehe\'{e}r, Sagdeev, and T\'{o}th~\cite{2024Cano} introduced the notion of the \emph{canonical Ramsey property} in Euclidean settings. Formally, a finite configuration \(X \) is said to have the canonical Ramsey property if there exists an integer \(n_0 = n_0(X)\) such that for any positive integer \(r\),
\[
\mathbb{E}^{n_0} \overset{r}{\rightarrow} (X; X)_{\mathrm{GR}}.
\]
They further extended the canonical results to all hypercubes, and to rectangles whose side lengths \(a\) and \(b\) satisfy \((\frac{a}{b})^2 \) is rational. They posed the natural question of whether this property holds for all rectangles.

In this paper, we resolve this conjecture affirmatively.

\begin{theorem}\label{thm:rectangle}
Let \(r\) be a positive integer, and let \(T\) be a rectangle. There exists an integer \(n_0 = n_0(T)\) such that for all \(n \ge n_0\),
\[
    \mathbb{E}^{n} \overset{r}{\rightarrow} (T; T)_{\mathrm{GR}}.
\]
\end{theorem}
While Theorem~\ref{thm:rectangle} extends previous results on squares and rational rectangles, our proof departs from prior methods. The key idea is to identify an auxiliary product structure within Euclidean space whose coloring involves only a bounded number of colors. This structural reduction enables the application of the celebrated simplex Ramsey theorem of Frankl and R\"{o}dl~\cite{1990JAMS} to control local configurations, and subsequently leverages product Ramsey properties~\cite{1973JCTA} to amplify these local patterns into a global monochromatic rectangle.

\section{The Proof}
We begin by introducing the notations used throughout the proof.

\begin{itemize}
    \item \( S_n(x) \): the point set of a regular \((n{-}1)\)-dimensional simplex of side length \(x\); that is, a set of \(n\) points \(\boldsymbol{x}_1, \dots, \boldsymbol{x}_n\) such that the Euclidean distance \(\|\boldsymbol{x}_i - \boldsymbol{x}_j\| = x\) for all \(i \neq j\).

    \item \( A \times B \): the Cartesian product of sets \(A\) and \(B\), defined as \(\{(\boldsymbol{a}, \boldsymbol{b}) : \boldsymbol{a} \in A,\, \boldsymbol{b} \in B\}\).

    \item \( \ell_x \): a two-point configuration consisting of a line segment of length \(x\).

    \item \( B_t(x, y) \): a planar configuration of \(t + 1\) points \(\boldsymbol{v}_1, \boldsymbol{v}_2, \ldots, \boldsymbol{v}_{t+1}\) lying in a common plane, such that
    \[
        \|\boldsymbol{v}_i - \boldsymbol{v}_{i+1}\| = y \quad \text{for all } 1 \le i \le t, \quad \text{and} \quad \|\boldsymbol{v}_1 - \boldsymbol{v}_{t+1}\| = x.
    \]
    That is, a polygonal path of \(a\) consecutive edges of length \(y\), whose endpoints are distance \(x\) apart. Here, it requires that \(t\ge\max\{2,\ceil{\frac{x}{y}}\}\). In particular, \(B_2(x, y)\) forms a triangle.

    \item For \(X \subseteq Y\), we write \(Y \overset{r}{\rightarrow} X\) if every \(r\)-coloring of \(Y\) contains a monochromatic configuration congruent to \(X\). In particular, a configuration \(X \subseteq \mathbb{E}^n\) is called \emph{Ramsey} if for all \(r\), \(\mathbb{E}^n \overset{r}{\rightarrow} X\) holds for some \(n = n(X, r)\).
    \item For \(K_{1},K_{2} \subseteq Y\), we write \(Y \overset{r}{\rightarrow} (K_1; K_2)_{\mathrm{GR}}\) if every \(r\)-coloring of \(Y\) contains a monochromatic copy of \(K_1\) or a rainbow congruent copy of \(K_2\).
\end{itemize}

We will use the following classical results.

\begin{theorem}[\cite{1990JAMS}]\label{thm:RodlFrankl}
Every simplex is Ramsey.
\end{theorem}

\begin{theorem}[\cite{1973JCTA}]\label{EGM}
If \(X\) and \(Y\) are Ramsey, then so is \(X \times Y\).
\end{theorem}

Next we provide the formal proof of~\cref{thm:rectangle}.

\begin{proof}[Proof of~\cref{thm:rectangle}]
Let \(T\) be a rectangle with side length \(x\) and \(y\), where \(x> y\). Let \(m := \left\lceil \frac{x}{y} \right\rceil\). By definitions, both \(S_7(y)\) and \(B_2(y, x)\) are simplices, and hence are Ramsey by Theorem~\ref{thm:RodlFrankl}. By Theorem~\ref{EGM}, their Cartesian product \(S_7(y) \times B_2(y, x)\) is also Ramsey. Therefore, for any integer \(q\), there exists a dimension \(n_0 = n_0(x, y, q)\) such that
\[
\mathbb{E}^{n_0} \overset{q}{\rightarrow} S_7(y) \times B_2(y, x).
\]
Set \(q = (m+1)\binom{3m+1}{2} + 3m + 1\), and fix such an \(n_0\). Define \(n := n_0 + 3m + 2\). Note that \(n\) depends only on the rectangle \(T\), and is independent of the number of colors \(r\). Let \(\chi \colon \mathbb{E}^{n} \to [r]\) be an arbitrary \(r\)-coloring. Observe that when \(r \le 3\), the statement follows trivially since every rectangle is Ramsey~\cite{1973JCTA}. Therefore, we may assume \(r \ge 4\) throughout the proof. We begin with the following key auxiliary lemma.

\begin{lemma}\label{lemma:aux1}
Let \(T\) be a rectangle with side length \(a\) and \(b\). Let \( r \ge 4 \) and \(s\ge \max \{2,\ceil{\frac{a}{b}}\}\) be integers. Then
\[
S_{3s+1}(a) \times B_{s}(a, b) \overset{r}{\rightarrow} (\ell_{a}; T)_{\textup{GR}},
\]
\end{lemma}
\begin{proof}[Proof of Lemma~\ref{lemma:aux1}]
Let \( \tau \colon S_{3s+1}(a) \times B_s(a, b) \to [r] \) be an arbitrary \(r\)-coloring of the product configuration. Suppose that there is no monochromatic copy of \(\ell_a\). We will show the existence of a rainbow copy of \(T\).

Recall that \( B_s(a, b) \) consists of \(s + 1\) points \(\boldsymbol{v}_1, \boldsymbol{v}_2, \ldots, \boldsymbol{v}_{s+1}\) lying in a common plane, such that
    \[
        \|\boldsymbol{v}_i - \boldsymbol{v}_{i+1}\| = b \quad \text{for all } 1 \le i \le s, \quad \text{and} \quad \|\boldsymbol{v}_1 - \boldsymbol{v}_{s+1}\| = a.
    \]
Fix any point \(\boldsymbol{x}_k \in S_{3s+1}(x)\), and consider the sequence
\[
(\boldsymbol{x}_k, \boldsymbol{v}_1),\ (\boldsymbol{x}_k, \boldsymbol{v}_2),\ \ldots,\ (\boldsymbol{x}_k, \boldsymbol{v}_{s+1})\in S_{3s+1}(a)\times B_{s}(a,b).
\]
By definitions, \(\|(\boldsymbol{x}_k, \boldsymbol{v}_1) - (\boldsymbol{x}_k, \boldsymbol{v}_{s+1})\| = a\). Since \(\tau\) avoids monochromatic copies of \(\ell_a\), these two points must have different colors:
\[
\tau(\boldsymbol{x}_k, \boldsymbol{v}_1) \ne \tau(\boldsymbol{x}_k, \boldsymbol{v}_{s+1}).
\]
Thus, there exists some index \(i \in \{1, \dots, s\}\) such that
\[
\tau(\boldsymbol{x}_k, \boldsymbol{v}_i) \ne \tau(\boldsymbol{x}_k, \boldsymbol{v}_{i+1}).
\]

Now, vary \(\boldsymbol{x}_k\) over all \(3s+1\) points in \(S_{3s+1}(a)\). For each \(k\in [3s+1]\), associate to it an index \(i_k\) where \(\tau(\boldsymbol{x}_k, \boldsymbol{v}_{i_k}) \ne \tau(\boldsymbol{x}_k, \boldsymbol{v}_{i_k+1})\). By the pigeonhole principle, there exists some fixed \(i\) such that
\[
\tau(\boldsymbol{x}_k, \boldsymbol{v}_i) \ne \tau(\boldsymbol{x}_k, \boldsymbol{v}_{i+1})
\]
holds for at least \(\ceil{\frac{3s+1}{s}}=4\) distinct values of \(k\). Without loss of generality, assume this holds for \(k = 1, 2, 3, 4\). We now examine the colors of the points \((\boldsymbol{x}_k, \boldsymbol{v}_i)\) and \((\boldsymbol{x}_k, \boldsymbol{v}_{i+1})\) for \(k = 1, 2, 3, 4\). Without loss of generality, we may assume
\[
\tau(\boldsymbol{x}_1, \boldsymbol{v}_i) = \text{red}, \quad \tau(\boldsymbol{x}_1, \boldsymbol{v}_{i+1}) = \text{blue}.
\]
For each fixed \(j \in \{i, i+1\}\), the set \(\{(\boldsymbol{x}_k, \boldsymbol{v}_j) : k = 1, 2, 3, 4\}\) forms a regular simplex with all pairwise distances equal to \(a\). Since \(\tau\) avoids monochromatic copies of \(\ell_a\), these four points must all receive distinct colors. In particular, the followings hold.
\begin{itemize}
    \item Among \(\tau(\boldsymbol{x}_k, \boldsymbol{v}_{i+1})\) for \(k = 1, 2, 3, 4\), at most one can be colored red.
    \item Among \(\tau(\boldsymbol{x}_k, \boldsymbol{v}_i)\) for \(k = 1, 2, 3, 4\), at most one can be colored blue.
\end{itemize}
Therefore, there exists some \(k \ne 1\) (without loss of generality, assume \(k = 2\)) such that
\[
\tau(\boldsymbol{x}_2, \boldsymbol{v}_i) \ne \text{blue}, \quad \tau(\boldsymbol{x}_2, \boldsymbol{v}_{i+1}) \ne \text{red}.
\]
Thus, the four vertices
\[
(\boldsymbol{x}_1, \boldsymbol{v}_i), \quad (\boldsymbol{x}_1, \boldsymbol{v}_{i+1}), \quad (\boldsymbol{x}_2, \boldsymbol{v}_i), \quad (\boldsymbol{x}_2, \boldsymbol{v}_{i+1})
\]
receive four distinct colors under \(\tau\). These form a rainbow copy of the rectangle \(T\) with side lengths \(a\) and \(b\), a contradiction. This completes the proof of~\cref{lemma:aux1}.
\end{proof}

 For each point \(\boldsymbol{u} \in \mathbb{E}^{n_0}\), consider the product configuration
\[
(\boldsymbol{u}, S_{3m+1}(x) \times B_m(x, y)) := \{(\boldsymbol{u}, \boldsymbol{v}) : \boldsymbol{v} \in S_{3m+1}(x) \times B_m(x, y)\} \subseteq \mathbb{E}^{n_0 + 3m + 2}.
\]
By Lemma~\ref{lemma:aux1}, under the coloring \(\chi\), each such set either contains a rainbow copy of \(T\), or a monochromatic copy of \(\ell_x\). If a rainbow copy of \(T\) occurs for some \(\boldsymbol{u}\), we are done. Thus, we may assume that for every \(\boldsymbol{u} \in \mathbb{E}^{n_0}\), no rainbow copy of \(T\) exists. Then Lemma~\ref{lemma:aux1} ensures that each associated product configuration must contain a monochromatic copy of \(\ell_x\). Since \(S_{3m+1}(x) \times B_m(x, y)\) is finite, there are only finitely many possible positions for such \(\ell_x\)'s. More precisely, each copy of \(\ell_x\) arises in one of the following two forms:
\begin{itemize}
    \item A line between two distinct points in \(S_{3m+1}(x)\), paired with a fixed point in \(B_m(x,y)\). There are \(\binom{3m+1}{2}\) choices of lines in \(S_{3m+1}(x)\), and for each such line, \(m+1\) possible choices of a point in \(B_m(x,y)\), yielding a total of \((m+1) \cdot \binom{3m+1}{2} \) copies.
    
    \item There exists one line in \(B_m(x,y)\) of length exactly \(x\), namely the line between \(v_1\) and \(v_{m+1}\), and pair it with a fixed point in \(S_{3m+1}(x)\). Since there are \(3m+1\) choices of points in \(S_{3m+1}(x)\), this contributes \(3m+1\) copies.
\end{itemize}
Therefore, the number of distinct copies of \(\ell_x\) in \(S_{3m+1}(x) \times B_m(x, y)\) is bounded by
\[
q= (m+1)\binom{3m+1}{2} + 3m+1.
\]
To formalize the selection of these monochromatic \(\ell_x\)'s, we define a labeling function
\[
\lambda : \{ \ell_x \subseteq S_{3m+1}(x) \times B_m(x,y) \} \to [q],
\]
which assigns a unique label to each possible position of \(\ell_x\).

We define an auxiliary coloring \(\gamma : \mathbb{E}^{n_0} \to [q]\) by the rule:
for each \(\boldsymbol{u} \in \mathbb{E}^{n_0}\), select one monochromatic copy of \(\ell_x\), denoted by \(\ell(\boldsymbol{u})\), within its associated product configuration under \(\chi\), and set
\[
\gamma(\boldsymbol{u}) := \lambda( \ell(\boldsymbol{u}) ).
\]
That is, \(\gamma(\boldsymbol{u})\) records the position label of the chosen monochromatic \(\ell_x\). By definitions of \(n_{0}\) and \(q\),
\[
\mathbb{E}^{n_0} \overset{q}{\rightarrow} S_7(y) \times B_2(y, x).
\]
Thus, there exists a monochromatic copy of \(S_7(y) \times B_2(y, x)\) under \(\gamma\). Fix this monochromatic copy and denote the set of its points as \(V \subseteq \mathbb{E}^{n_0}\).

By the definition of \(\gamma\), every point \(\boldsymbol{u} \in V\) is associated with the same monochromatic copy of \(\ell_x\), whose endpoints \(\boldsymbol{v}_1, \boldsymbol{v}_2\) lie in \(S_{3m+1}(x) \times B_m(x,y)\). That is, for all \(\boldsymbol{u} \in V\), the pair \(\{ (\boldsymbol{u}, \boldsymbol{v}_1), (\boldsymbol{u}, \boldsymbol{v}_2) \}\) forms a monochromatic copy of \(\ell_x\) under \(\chi\). Consider now the set of points
\[
V \times \{ \boldsymbol{v}_1, \boldsymbol{v}_2 \} \subseteq \mathbb{E}^{n_0 + 3m + 2}.
\]
If this set contains a rainbow copy of \(T\) under \( \chi \), we are done. Otherwise, focus on the subset
\[
(V, \boldsymbol{v}_1) := \{ (\boldsymbol{u}, \boldsymbol{v}_1) : \boldsymbol{u} \in V \}.
\]
Since \(V\) is isomorphic to \(S_7(y) \times B_2(y, x)\), we can apply Lemma~\ref{lemma:aux1} again with parameter \(s = 2\). Here, we interpret the rectangle \(T\) as being rotated by 90 degrees, effectively treating the side of length \(y\) as the dominant side in the product structure. This flexible application of Lemma~\ref{lemma:aux1} allows us to leverage the same combinatorial argument with the roles of \(x\) and \(y\) interchanged. More precisely, if a rainbow copy of \(T\) under \(\chi\) is found, we are done; otherwise, Lemma~\ref{lemma:aux1} guarantees the existence of a monochromatic copy of \(\ell_y\), say between the points \((\boldsymbol{u}_1, \boldsymbol{v}_1)\) and \((\boldsymbol{u}_2, \boldsymbol{v}_1)\).

Finally, we make the following observations under the coloring function \(\chi\):
\begin{itemize}
    \item The points \((\boldsymbol{u}_1, \boldsymbol{v}_1)\) and \((\boldsymbol{u}_2, \boldsymbol{v}_1)\) share the same color, as they form a monochromatic copy of \(\ell_y\) constructed in the previous step. 
    \item The points \((\boldsymbol{u}_1, \boldsymbol{v}_1)\) and \((\boldsymbol{u}_1, \boldsymbol{v}_2)\) also share the same color, since the line with endpoints \(\boldsymbol{v}_1, \boldsymbol{v}_2\) was fixed as a monochromatic \(\ell_x\) across all associated product configurations. Likewise, the points \((\boldsymbol{u}_2, \boldsymbol{v}_1)\) and \((\boldsymbol{u}_2, \boldsymbol{v}_2)\) are colored identically.
\end{itemize}
Therefore, the four points \( (\boldsymbol{u}_1, \boldsymbol{v}_1),  (\boldsymbol{u}_1, \boldsymbol{v}_2), (\boldsymbol{u}_2, \boldsymbol{v}_1), (\boldsymbol{u}_2, \boldsymbol{v}_2)\) form a monochromatic rectangle congruent to \(T\) under the original coloring \(\chi\). This completes the proof.
\end{proof}

\section{Concluding remarks}
In this paper, we have completely resolved the canonical Ramsey property for rectangles, thereby answering the open problem posed by Geh\'{e}r, Sagdeev, and T\'{o}th~\cite{2024Cano}. Our proof establishes that for any rectangle \(T\), there exists a dimension \(n_0 = n_0(T)\) such that every \(r\)-coloring of \(\mathbb{E}^{n_0}\) contains either a monochromatic or a rainbow congruent copy of \(T\), with \(n_0\) depending solely on the parameters of \(T\) and independent of the number of colors \(r\). While the bound on \(n_0\) obtained through our construction is admittedly large, optimizing this dimension remains an intriguing problem. Nevertheless, \(n_0\) cannot, in general, be reduced to 2; see~\cite[Theorem~1.3]{2022arxivEGR} for a precise obstruction.

A natural and challenging direction for future research is to determine whether this phenomenon extends to all \emph{rectangular sets}, finite configurations that are subsets of the vertex set of a rectangular box in \(\mathbb{E}^n\). Proving the canonical Ramsey property for these sets would constitute a substantial leap beyond the current case of full rectangles. A more precise problem is to establish whether the canonical Ramsey property holds for all parallelograms in the plane. A positive resolution would, in particular, imply a complete solution for arbitrary triangles. The general triangle case remains one of the most prominent open problems in Euclidean Gallai-Ramsey theory~\cite{2024Cano}. Partial progress has been made: Cheng and Xu~\cite{2025DCGChengXu} proved the canonical Ramsey property for all acute triangles, right triangles, and certain obtuse triangles whose obtuse angle is not too large. Extending these results to arbitrary triangles will likely require fundamentally new structural insights and techniques beyond the current framework.

\bibliographystyle{abbrv}
\bibliography{gallairamsey}
\end{document}